\documentclass[10pt,]{article} 
\usepackage[numbers]{natbib}
\usepackage[english]{babel}
\usepackage{bbm}
\usepackage{amsmath}         		
\usepackage{amssymb}
\usepackage{amsthm}
\usepackage{thmtools,thm-restate}
\usepackage[left=1in,right=1in,top=1in,bottom=1in]{geometry}
\usepackage{tikz}
\usetikzlibrary{shapes}
\usepackage{float}
\usepackage{url}

\include{auxfiles}

\usepackage{stmaryrd}

\newcommand{\set}[1]{\left\{ #1 \right\}}
\newcommand{\inner}[2]{\left\langle #1,\, #2 \right\rangle}
\newcommand{\norm}[1]{\left\| #1 \right\|}
\DeclareMathOperator*{\argmax}{arg\,max}

\newcommand{\abs}[1]{\left| #1 \right|}
\newcommand{\spn}[1]{\text{span}\set{ #1 }}

\renewcommand{\abstract}[1]{
	\centerline{
	\begin{minipage}{0.8\textwidth}
	\begin{centering}
	\textbf{Abstract}\\
	#1 
	\end{centering}
	\end{minipage}
	}
}

\newcommand{\normediter}[1]{\frac{x^{( #1 )}}{\norm{x^{( #1 )}}}}
\newcommand{\normalized}[1]{\frac{#1}{\norm{#1}}}

\newtheorem{theorem}{Theorem}
\newtheorem{lemma}{Lemma}

\title{\rule{6.5in}{2pt}\\Lower Bound for Randomized First Order Convex
  Optimization\\\rule[2mm]{6.5in}{0.5pt}}
\author{}
\date{}
\begin{document}
\maketitle
\vspace{-21mm}
\textbf{Blake Woodworth} \hfill \url{blake@ttic.edu}\\
\textbf{Nathan Srebro} \hfill \url{nati@ttic.edu}\\
{\small Toyota Technological Institute at Chicago, Chicago, IL 60637, USA}

\vspace{5mm}

\abstract{We provide an explicit construction and direct proof for the
  lower bound on the number of first order oracle accesses required
  for a {\em randomized } algorithm to minimize a convex Lipschitz
  function.}

\section{Introduction}
We prove lower bounds for the complexity of first-order optimization using a randomized algorithm for the following problem:
\begin{equation} \label{eq:theproblem}
\min_{x \in \mathbb{R}^d : \norm{x}\leq B} f(x)
\end{equation}
where $f$ is convex and $L$-Lipschitz continuous with respect to the 
Euclidean norm. We consider
a standard oracle access model: at each iteration the
algorithm selects, possibly at random, a vector $x \in \mathbb{R}^d, \norm{x}
\leq B$ based on the oracle's responses to previous queries. The oracle then 
returns the function value $f(x)$ and some
subgradient $g \in \partial f(x)$ chosen by the oracle.  We bound
the expected number of iterations $T$, as a function of $L$,$B$ and
$\epsilon$, needed to ensure that for any convex $L$-Lipschitz
function, any valid first order oracle, and any dimension,
$f(x_T) \leq \min_{x \in \mathbb{R}^d : \norm{x}\leq B} f(x) + \epsilon$.
We are interested in lower bounding dimension-independent performance 
(i.e.~what an algorithm can guarantee in an arbitrary dimension)
and in our constructions we allow the dimension to grow as $\epsilon\rightarrow 0$.

\citet{NemirovskiYudin} carefully study both randomized and
deterministic first order optimization algorithms and give matching
upper and lower bounds for both, establishing a tight worst-case
complexity for \eqref{eq:theproblem}, whether using randomized or
deterministic algorithms, of $\Theta(L^2B^2/\epsilon^2)$ oracle
queries.  

The lower bound for deterministic algorithms is fairly direct,
well-known and has been reproduced in many forms in books,
tutorials, and lecture notes in the ensuing four decades, as are the
lower bounds for algorithms (whether randomized or deterministic)
where the iterates $x_t$ are constrained to be in the span of previous
oracle responses.  When the iterates are constrained to be in this
span, one can ensure the first $t+1$ iterates are spanned by the first
$t$ standard basis vectors $e_1,\ldots,e_t$, and that no point in this span can
be $O(1/\sqrt{t})$-suboptimal.  For deterministic algorithms, even if
the iterates escape this span, one can adversarialy rotate the
objective function so that the algorithm only escapes in useless
directions, obtaining the exact same lower bound using a very similar
construction.  Either way, a dimensionality of
$d=\Theta(T)=\Theta(L^2B^2/\epsilon^2)$ is sufficient to construct a
function requiring $\Theta(L^2B^2/\epsilon^2)$ queries to optimize.

Analyzing randomized algorithms which are allowed to leave the span of
oracle responses is trickier: the algorithm may guess directions, and
since even if we know the algorithm, we do not know in advance which
directions it will guess, we cannot rotate the function so as to avoid
these directions.  \citet{NemirovskiYudin} do provide a detailed and
careful analysis for such randomized algorithms, using a recursive
reduction argument and without a direct construction.  To the best of
our knowledge, this lower bound has not since been simplified, and so
lower bounds for randomized algorithms are rarely if ever covered in
books, tutorials and courses.  In this note, we provide an explicit
construction establishing the following lower bound:

% \begin{theorem} \label{thm:LB} For any $L,B,\delta > 0$, $\epsilon \in
%   (0,\frac{LB}{2})$, and dimension $d \geq
%   \frac{L^8B^8}{\epsilon^8}\log\frac{L^4B^4}{\epsilon^4\delta}$, there
%   exists a distribution over convex $L$-Lipschitz functions $f: \set{x
%     \in \mathbb{R}^d : \norm{x}\leq B} \mapsto \mathbb{R}$ and an
%   appropriate first order oracle such that, with probability
%   $1-\delta$, any randomized algorithm must make $\Omega(L^2
%   B^2/\epsilon^2)$ queries to the oracle in order to find an
%   $\epsilon$-suboptimal point.
% \end{theorem}

% $d \geq \frac{1}{\epsilon^8}\log\frac{1}{\epsilon^4}$

\begin{theorem} \label{thm:LB} For any $L,B$, $\epsilon \in
  (0,\frac{LB}{2})$, dimension $d \geq
  \frac{2L^8B^8}{\epsilon^8}\log\frac{L^4B^4}{\epsilon^4}$, and
  any randomized optimization algorithm, there exists a convex 
  $L$-Lipschitz function $f: \set{x \in \mathbb{R}^d : \norm{x}\leq B} 
  \mapsto \mathbb{R}$ and an appropriate first order oracle such 
  that the algorithm must make $\Omega(L^2B^2/\epsilon^2)$ queries 
  to the oracle in expectation in order to find an
  $\epsilon$-suboptimal point.
\end{theorem}

Our construction and proof directly captures the following intuition:
if the dimension is large enough, blindly guessing a direction becomes
increasingly difficult, and the algorithm should not gain much by such
random guessing.  In the standard construction used for the
deterministic lower bound, guessing a direction actually does provide
information on all useful directions.  However, by slightly perturbing the
standard construction, we are able to avoid such information leakage.
To do so, we use a technique we recently developed in order to analyze
finite sum structured optimization problems \cite{Woodworth16}.

In this note we only consider Lipschitz (non-smooth) functions without
an assumption of strong convexity.  A reduction or simple modification
to the construction can be used to establish a lower bound for
Lipschitz (non-smooth) strongly convex functions.  Applying the same
technique we use here to the standard lower bound construction for
smooth functions leads to lower bounds for randomized
algorithms for smooth non-strongly-convex and smooth strongly
convex first order optimization too.  All of these lower bounds would
match those for deterministic optimization, with a polynomial increase
in the dimension required.  This polynomial increase can likely
be reduced to a smaller polynomial through more careful
analysis.

Theorem \ref{thm:LB} (which we reiterate also follows from the
more detailed analysis of \citeauthor{NemirovskiYudin}) shows that
randomization cannot help first order optimization.  It
is important to emphasize that this should not be taken for granted, and
that in other situations randomization {\em could} be beneficial.
For example, when optimizing finite sum structured objectives, randomization provably
reduces the oracle complexity \cite{Woodworth16}.  It is thus
important to specifically and carefully consider randomized algorithms
when proving oracle lower bounds, and we hope this note will aid in
such analysis.

We would like to thank the authors of \cite{Carmon17} Yair Carmon, John Duchi, Oliver Hinder, and Aaron Sidford for pointing out a mistake with our original proof of Lemma \ref{lem:1term}, which has since been corrected. 

\section{Proof of Theorem \ref{thm:LB}}
Without loss of generality assume $L = B = 1$. Consider a family of functions $\mathcal{F}$ of the form
\begin{equation}
f(x) = \max_{1\leq j \leq k} \left(\inner{x}{v_j} - j c \right)
\end{equation} 
where $k = \frac{1}{4\epsilon^2}$, $c = \frac{\epsilon}{k}$, and the vectors $v_j$ are an orthonormal set in $\mathbb{R}^d$. Each of these functions is the maximum of linear functions thus convex and $1$-Lipschitz.
Drawing the orthonormal set of vectors $v_j$ uniformly at random specifies a distribution over the family of functions $\mathcal{F}$. Our approach will be to show that any \emph{deterministic} optimization algorithm must make at least $\Omega(1/\epsilon^2)$ oracle queries in expectation over the randomness in the choice of $f$. This implies through Yao's minimax principle a lower bound on the expected number of queries needed by a randomized algorithm on the worst-case function in $\mathcal{F}$. Therefore, for the remainder of the proof we need only consider deterministic optimization algorithms and functions drawn from this distribution over $\mathcal{F}$.

First, we show that minimizing a given function $f$ amounts to finding a vector $x$ which has significant negative correlation with \emph{all} of the vectors $v_j$. Consider the unit vector $\hat{x} = -\frac{1}{\sqrt{k}}\sum_{j=1}^k v_j$
\begin{equation}
f(\hat{x}) 
= \max_{1\leq j \leq k} \left( \inner{\hat{x}}{v_j} - j c \right) 
= -\frac{1}{\sqrt{k}} - c 
= -2\epsilon - c
\geq f(x^*)
\end{equation}
Therefore, for any $x$ such that $\inner{x}{v_j} > -\frac{c}{2}$ for some $j$, 
\begin{equation} \label{eq:small_ip_suboptimal}
f(x) 
\geq \inner{x}{v_j} - jc 
> -\frac{c}{2} - kc 
= -\frac{c}{2} - \epsilon 
> f(\hat{x}) + \epsilon 
\geq f(x^*) + \epsilon
\end{equation}
Consequently, any such $x$ cannot be $\epsilon$-suboptimal. Therefore, in order to show that the expected number of oracle queries is $\Omega(k) = \Omega(1/\epsilon^2)$, it suffices to show that the following event occurs with constant probability:
\begin{equation}
E = \left\llbracket\ \forall t \leq k\ \forall j \geq t\ \abs{\inner{x^{(t)}}{v_j}} < \frac{c}{2}\ \right\rrbracket
\end{equation}

Let $S_t = \spn{x^{(1)},...,x^{(t)}, v_1, ..., v_{t}}$ and let $S_t^\perp$ be its orthogonal complement. Let $P_t$ and $P^\perp_t$ be (orthogonal) projection operators onto $S_t$ and $S_t^\perp$ respectively. Consider the events
\begin{equation}
G_t = \left\llbracket\ \forall j \geq t\ \abs{\inner{\normalized{P_{t-1}^{\perp}x^{(t)}}}{v_j}} < \frac{c}{2(\sqrt{2}+\sqrt{k-1})}\ \right\rrbracket
\end{equation}
These events are useful because:

\begin{lemma} \label{lem:GimpE}
$\bigcap_{t=1}^k G_t \implies E$
\end{lemma}
\begin{proof}
Let $G_{\leq t}$ denote $\bigcap_{t'=1}^t G_{t'}$. 
It suffices to show that for each $t\leq k$, $G_{\leq t} \implies \forall j \geq t\  \abs{\inner{x^{(t)}}{v_j}} < \frac{c}{2}$.
For each $t \leq k$ and $j \geq t$
\begin{equation}
\begin{aligned} \label{eq:GimpE1}
\abs{\inner{x^{(t)}}{v_j}} 
&\leq \norm{x^{(t)}}\abs{\inner{\normediter{t}}{P_{t-1}v_j}} + \norm{x^{(t)}}\abs{\inner{\normediter{t}}{P_{t-1}^{\perp}v_j}} \\
&\leq \norm{P_{t-1}v_j} + \abs{\inner{\frac{P_{t-1}^{\perp}x^{(t)}}{\norm{x^{(t)}}}}{v_j}} \\
&\leq \norm{P_{t-1}v_j}  + \abs{\inner{\normalized{P_{t-1}^{\perp}x^{(t)}}}{v_j}} \\
&\leq \norm{P_{t-1}v_j}  + \frac{c}{2(\sqrt{2} + \sqrt{k-1})}
\end{aligned}
\end{equation}
First, we decomposed $v_j$ into its $S_{t-1}$ and $S_{t-1}^\perp$ components and applied the triangle inequality. Next, we used that $\norm{x^{(t)}} \leq 1$ and that the orthogonal projection operator $P_{t-1}^\perp$ is self-adjoint. Finally, we used that the projection operator is non-expansive and then applied the definition of $G_t$.

Next, we will prove by induction on $t$ that for all $t \leq k$ and all $j \geq t$, $G_{<t} \implies \norm{P_{t-1}v_j}^2 \leq \frac{c^2(t-1)}{2(\sqrt{2}+\sqrt{k-1})^2}$. The case $t=1$ is trivial since the left hand side is the projection of $v_j$ onto the empty set. 

For the inductive step, fix any $t \leq k$ and $j \geq t$. Let $\hat{P}_t$ project onto $\spn{x^{(1)},...,x^{(t+1)},v_1,...,v_{t}}$ (this includes $x^{(t+1)}$ in contrast with $P_t$) and let $\hat{P}_t^\perp$ be the projection onto the orthogonal subspace.
Since $\set{x^{(1)},...,x^{(t-1)}, v_1,...,v_{t-1}}$ spans $S_{t-1}$, the Gram-Schmidt vectors
\begin{equation}
\normalized{P_0^\perp x^{(1)}}, \normalized{\hat{P}_0^\perp v_1}, \normalized{P_1^\perp x^{(2)}}, \normalized{\hat{P}_1^\perp v_2}, ..., \normalized{P_{t-2}^\perp x^{(t-1)}}, \normalized{\hat{P}_{t-2}^\perp v_{t-1}}
\end{equation}
are an orthonormal basis for $S_{t-1}$ (after ignoring any zero vectors that arise from projection).

We now write $\norm{P_{t-1}v_j}$ in terms of this orthonormal basis:
\begin{equation}
\begin{aligned} \label{eq:GimpE2}
\norm{P_{t-1} v_j}^2 
&= \sum_{i=1}^{t-1} \inner{\normalized{P_{i-1}^\perp x^{(i)}}}{v_j}^2
       + \sum_{i=1}^{t-1} \inner{\normalized{\hat{P}_{i-1}^\perp v_i}}{v_j}^2 \\
&\leq \frac{c^2(t-1)}{4(\sqrt{2} + \sqrt{k-1})^2} + \sum_{i=1}^{t-1} \frac{1}{\norm{\hat{P}_{i-1}^\perp v_i}^2}\inner{\hat{P}_{i-1}^\perp v_i}{v_j}^2
\end{aligned}
\end{equation}
The inequality follows from the definition of $G_{<t}$. We must now bound the second term of \eqref{eq:GimpE2}. Focusing on the inner product one individual term in the sum
\begin{equation}\label{eq:rewriteprojection}
\begin{aligned} 
\abs{\inner{\hat{P}_{i-1}^\perp v_i}{v_j}} 
&= \abs{\inner{v_i}{v_j} - \inner{\hat{P}_{i-1}v_i}{v_j}} \\
&= \abs{\inner{P_{i-1}v_i}{v_j} + \inner{\inner{\normalized{P_{i-1}^\perp x^{(i)}}}{v_i}\normalized{P_{i-1}^\perp x^{(i)}}}{v_j}} \\
&\leq \abs{\inner{P_{i-1}v_i}{P_{i-1}v_j}} + \abs{\inner{\normalized{P_{i-1}^\perp x^{(i)}}}{v_i}\inner{\normalized{P_{i-1}^\perp x^{(i)}}}{v_j}}
\end{aligned}
\end{equation}
By the Cauchy-Schwarz inequality and the inductive hypothesis, the first term is bounded by $\norm{P_{i-1}v_i}\norm{P_{i-1}v_j} \leq \frac{c^2(i-1)}{2(\sqrt{2}+\sqrt{k-1})^2}$. By the definition of $G_{<t}$, the second term is bounded by $\frac{c^2}{4(\sqrt{2}+\sqrt{k-1})^2}$.
Furthermore, by our choice of $c = \frac{\epsilon}{k}$ and $k = \frac{1}{4\epsilon^2}$, $c = \frac{1}{2\sqrt{k}}$ so we conclude that
\begin{equation} \label{eq:GimpE2summand}
\begin{aligned}
\abs{\inner{\hat{P}_{i-1}^\perp v_i}{v_j}}
&\leq \frac{c^2(2i-1)}{4(\sqrt{2}+\sqrt{k-1})^2} \\
&= \frac{c}{4(\sqrt{2}+\sqrt{k-1})} \frac{2i-1}{2\sqrt{k}(\sqrt{2}+\sqrt{k-1})} \\
&\leq \frac{c}{4(\sqrt{2}+\sqrt{k-1})} \frac{2i-1}{2k} \\
&\leq \frac{c}{4(\sqrt{2}+\sqrt{k-1})}
\end{aligned}
\end{equation}

We have now upper bounded the inner products in \eqref{eq:GimpE2}, it remains to lower bound the norm in the denominator. Rewriting the projection $\hat{P}_{i-1}^\perp$ as in \eqref{eq:rewriteprojection}:
\begin{equation}
\begin{aligned}
\norm{\hat{P}_{i-1}^\perp v_i}^2 
&= \abs{\inner{\hat{P}_{i-1}^\perp v_i}{v_i}} \\
&= 1 - \norm{P_{i-1}v_i}^2 - \inner{\normalized{P_{i-1}^\perp x^{(i)}}}{v_i}^2 \\
&\geq 1 - \frac{c^2(i-1)}{2(\sqrt{2}+\sqrt{k-1})^2} - \frac{c^2}{4(\sqrt{2}+\sqrt{k-1})^2}
\end{aligned}
\end{equation}
This quantity is at least $\frac{1}{2}$ because $c = \frac{1}{2\sqrt{k}} < 1$ so
\begin{equation}
\begin{aligned}
\frac{c^2(i-1)}{2(\sqrt{2}+\sqrt{k-1})^2} + \frac{c^2}{4(\sqrt{2}+\sqrt{k-1})^2} 
&< \frac{2i-1}{(\sqrt{2}+\sqrt{k-1})^2} \\
&= \left(\frac{\sqrt{i-\frac{1}{2}}}{\sqrt{2}+\sqrt{k-1}}\right)^2 \\
&< \frac{1}{2} \left( \sqrt{\frac{i - \frac{1}{2}}{k}} \right)^2\\
&\leq \frac{1}{2}
\end{aligned}
\end{equation}
Combining this with \eqref{eq:GimpE2summand} and returning to \eqref{eq:GimpE2} we have that
\begin{equation}
\begin{aligned}
\norm{P_{t-1} v_j}^2 
&\leq \frac{c^2(t-1)}{4(\sqrt{2} + \sqrt{k-1})^2} + \sum_{i=1}^{t-1} \frac{\inner{\hat{P}_{i-1}^\perp v_i}{v_j}^2}{\norm{\hat{P}_{i-1}^\perp v_i}^2} \\
&\leq \frac{c^2(t-1)}{4(\sqrt{2} + \sqrt{k-1})^2} + \sum_{i=1}^{t-1} 4\left(\frac{c}{4(\sqrt{2}+\sqrt{k-1})}\right)^2 \\
&= \frac{c^2(t-1)}{4(\sqrt{2} + \sqrt{k-1})^2} + \frac{c^2(t-1)}{4(\sqrt{2} + \sqrt{k-1})^2} = \frac{c^2(t-1)}{2(\sqrt{2} + \sqrt{k-1})^2}
\end{aligned}
\end{equation}
which completes the inductive step. Finally, we return to \eqref{eq:GimpE1} and conclude that
\begin{equation}
\begin{aligned}
\abs{\inner{x^{(t)}}{v_j}} 
&\leq \norm{P_{t-1}v_j}  + \frac{c}{2(\sqrt{2} + \sqrt{k-1})} \\
&\leq \sqrt{\frac{c^2(t-1)}{2(\sqrt{2} + \sqrt{k-1})^2}} + \frac{c}{2(\sqrt{2} + \sqrt{k-1})} \\
&= \frac{c(\sqrt{2} + \sqrt{t-1})}{2(\sqrt{2} + \sqrt{k-1})} \leq \frac{c}{2}
\end{aligned}
\end{equation}
This completes the proof.
\end{proof}

In our model of computation, the oracle can provide the algorithm with any subgradient at the query point. We can therefore design a resisting oracle which returns subgradients that are as uninformative as possible. At a given point $x$, the subdifferential of $f$ is
\begin{equation}
\partial f(x) = \textrm{Conv}\set{v_\ell\ :\ \ell \in \argmax_{1\leq j \leq k} \left(\inner{x}{v_j} - j c \right)}
\end{equation}
and our resisting oracle will return as a subgradient
\begin{equation}
v_{\ell}\quad\text{where}\quad \ell = \min\set{\ \argmax_{1\leq j \leq k} \left(\inner{x}{v_j} - j c \right)\ }
\end{equation}
That is, the returned subgradient will always be a single vector $v_\ell$ for the smallest value of $\ell$ that corresponds to a valid subgradient.

\begin{lemma} \label{lem:oracleresponses}
For each $t \leq k$, let $g^{(t)} \in \partial f\left(x^{(t)}\right)$ be the subgradient returned by the oracle. Then $G_{\leq t} \implies g^{(t)} \in \set{v_1,...,v_t}$.
\end{lemma}
\begin{proof}
This follows from the structure of the objective function $f$ and our choice of subgradient oracle. In the proof of Lemma \ref{lem:GimpE}, we established that $G_{\leq t} \implies \forall j \geq t\ \abs{\inner{x^{(t)}}{v_j}} < \frac{c}{2}$. Thus for any $j > t$
\begin{equation}
\inner{x^{(t)}}{v_t} - tc > -\frac{c}{2} - tc = \frac{c}{2} - (t+1)c > \inner{x^{(t)}}{v_j} - jc 
\end{equation}
Therefore, no $j > t$ can index a maximizing term in $f$ so $g^{(t)} \subseteq \set{v_1,...,v_t}$.
\end{proof}

By Lemma \ref{lem:GimpE} and the chain rule of probability
\begin{equation} \label{eq:chainrule}
\mathbb{P}\left[ E \right] \geq \mathbb{P}\left[ \bigcap_{t=1}^k G_t \right] = \prod_{t=1}^k \mathbb{P}\left[ G_t\ \middle|\ G_{<t} \right]
\end{equation}
Focusing on a single term in the product:

\begin{lemma} \label{lem:1term}
For any $t \leq k$, $\mathbb{P}\left[ G_t\ \middle|\ G_{<t} \right] > 1 - (k-t+1)\exp\left( \frac{-c^2(d-2t+1)}{8(\sqrt{2}+\sqrt{k-1})^2}\right)$
\end{lemma}
\begin{proof}
The key to lower bounding $\mathbb{P}\left[ G_t\ \middle|\ G_{<t} \right]$ is to show that for $j \geq t$ the vector $\normalized{P_{t-1}^\perp v_j}$ is uniformly distributed on the unit sphere in $S_{t-1}^\perp$ conditioned on $G_{<t}$ and $\set{v_1,...,v_{t-1}}$. If we can show this, then the inner product in the definition of $G_t$ is effectively between a fixed vector and a random unit vector, and the probability that this is large decreases rapidly as the dimension grows.

Fix an arbitrary $t \leq k$ and $j \geq t$. Let $V_{<t} := \set{v_1, ...,v_{t-1}}$ be any set of orthonormal vectors in $\mathbb{R}^d$. We will show that the density $p_{V_{\geq t}}\left(V_{\geq t}\ \middle|\ G_{<t}, V_{<t}\right)$ is invariant under rotations which preserve $\set{x^{(1)},...,x^{(t)}, v_1,...,v_{t-1}}$.

Let $R$ be any rotation $R^TR = I_{d\times d}$ such that $\forall w\in\spn{x^{(1)},...,x^{(t-1)}, v_1,...,v_{t-1}}$ $Rw = R^Tw = w$. We will show that $p_{V_{\geq t}}\left(V_{\geq t}\ \middle|\ G_{<t}, V_{<t}\right) = p_{V_{\geq t}}\left(RV_{\geq t}\ \middle|\ G_{<t}, V_{<t}\right)$. To begin
\begin{equation}
p_{V_{\geq t}}\left(V_{\geq t}\ \middle|\ G_{<t}, V_{<t}\right) =
\frac{\mathbb{P}\left( G_{<t}\ \middle|\ V \right)p_V(V)}{\mathbb{P}\left( G_{<t}\ \middle|\ V_{<t} \right)p_{V_{<t}}(V_{<t})}
\end{equation}
and
\begin{equation}
p_{V_{\geq t}}\left(RV_{\geq t}\ \middle|\ G_{<t}, V_{<t}\right) =
\frac{\mathbb{P}\left( G_{<t}\ \middle|\ V \right)p_V(RV)}{\mathbb{P}\left( G_{<t}\ \middle|\ V_{<t} \right)p_{V_{<t}}(V_{<t})}
\end{equation}
Since $V = \set{v_1,...,v_k}$ is marginally distributed uniformly, $p_V(V) = p_V(RV)$, so it only remains to show that $\mathbb{P}\left( G_{<t}\ \middle|\ V \right) = \mathbb{P}\left( G_{<t}\ \middle|\ RV \right)$. Recall that at this time we are considering an arbitrary \emph{deterministic} algorithm minimizing a randomly selected $f$. Thus for any particular $V$, which fixes $f$, either $G_{<t}$ holds or it does not--so the probabilities are either 0 or 1. 

We will show by induction that for every $i < t$, if $\mathbb{P}\left( G_{<t}\ \middle|\ V \right) = 1$ then $\mathbb{P}\left( G_{<t}\ \middle|\ RV \right) = 1$ too. The case $i = 1$ is trivial since $G_{<1}$ is independent of $V$. Consider now some $1 < i < t$, and suppose that $\mathbb{P}\left( G_{<i}\ \middle|\ V \right) = 1$. Since $G_{<i} \implies G_{<s}$ for $s < i$, $\mathbb{P}\left( G_{<s}\ \middle|\ V \right) = 1$ for all $s \leq i$ and by the inductive hypothesis $\mathbb{P}\left( G_{<s}\ \middle|\ RV \right) = 1$ for all $s < i$. Thus, it just remains to show that $P(G_{i-1}\ |\ G_{<i-1}, RV) = 1$. Let $P'_{i}$ be the projection operator onto $\set{x'^{(1)},...,x'^{(i)},Rv_1,...,Rv_i}$ where the $x'$ are the oracle queries made by the algorithm when $f$ is determined by $RV$. For any $\ell \geq i-1$, consider
$\abs{\inner{\normalized{{P'}_{i-2}^\perp {x'}^{(i-1)}}}{Rv_\ell}}$.

Since $G_{<i-1}$ holds when $f$ is determined by $V$, by Lemma \ref{lem:oracleresponses} the queries $\set{x^{(1)},...,x^{(i-1)}}$ are determined by $\set{v_1,...,v_{i-2}}$. Since $G_{<i-1}$ holds when $f$ is determined by $RV$ and $\set{v_1,...,v_{i-2}} = \set{Rv_1,...,Rv_{i-2}}$, ${x'}^{(i-1)} = x^{(i-1)}$. Furthermore, since $R$ preserves $\set{x^{(1)},...,x^{(i-2)},v_1,...,v_{i-2}}$, it is also the case that ${P'}_{i-2}^\perp = P_{i-2}^\perp$.
Finally, since $P_{i-2}^\perp x^{(i-1)} = x^{(i-1)} - P_{i-2} x^{(i-1)} \in \spn{x^{(1)},...,x^{(i-1)},v_1,...,v_{i-2}}$, it is unchanged by $R^T$, therefore
\begin{equation}
\abs{\inner{\normalized{{P'}_{i-2}^\perp {x'}^{(i-1)}}}{Rv_\ell}} 
= \abs{\inner{R^T\normalized{P_{i-2}^\perp x^{(i-1)}}}{v_\ell}}
= \abs{\inner{\normalized{P_{i-2}^\perp x^{(i-1)}}}{v_\ell}}
< \frac{c}{2(\sqrt{2}+\sqrt{k-1})}
\end{equation}
since $G_{i-1}$ holds when $f$ is determined by $V$. Therefore, we conclude that $p_{V_{\geq t}}\left(V_{\geq t}\ \middle|\ G_{<t}, V_{<t}\right)$ is invariant under rotations that preserve $\set{x^{(1)},...,x^{(t-1)},v_1,...,v_{t-1}}$.

For a given $j \geq t$, the marginal density of $v_j$ conditioned on $G_{<t}, V_{<t}$ is invariant under $R$. By Lemma \ref{lem:oracleresponses}, since the optimization algorithm is deterministic, the queries $\set{x^{(1)},...,x^{(t)}}$ are completely determined given $G_{<t}, V_{<t}$, thus the projection $P_{t-1}^\perp$ is also determined by $G_{<t}, V_{<t}$. Therefore, the random vectors $\normalized{P_{t-1}^\perp v_j}$ and $\normalized{P_{t-1}^\perp Rv_j}$ have the same density. The rotation $R$ preserves $S_{t-1}$ and $R$ preserves length, so $\normalized{P_{t-1}^\perp Rv_j} = R\normalized{P_{t-1}^\perp v_j}$ and we conclude that the distribution of $\normalized{P_{t-1}^\perp v_j}$ conditioned on $G_{<t}, V_{<t}$ is spherically symmetric on $S_{t-1}^\perp$. 

We can now lower bound $\mathbb{P}\left[ G_t\ \middle|\ G_{<t} \right] = \mathbb{E}_{V_{<t}}\left[\mathbb{P}\left[ G_t\ \middle|\ G_{<t}, V_{<t} \right]\right] \geq \inf_{V_{<t}} \mathbb{P}\left[ G_t\ \middle|\ G_{<t}, V_{<t} \right]$.
For any $V_{<t}$,
\begin{equation}
\begin{aligned}
\mathbb{P}\left[ G_t\ \middle|\ G_{<t}, V_{<t} \right]
&= \mathbb{P}\left[\forall j \geq t\ \abs{\inner{\normalized{P_{t-1}^{\perp}x^{(t)}}}{v_j}} < \frac{c}{2(\sqrt{2}+\sqrt{k-1})} \ \middle|\ G_{<t}, V_{<t} \right] \\
&\geq 1- \sum_{j=t}^k \mathbb{P}\left[\abs{\inner{\normalized{P_{t-1}^{\perp}x^{(t)}}}{v_j}} \geq \frac{c}{2(\sqrt{2}+\sqrt{k-1})} \ \middle|\ G_{<t}, V_{<t} \right] \\
&\geq 1- \sum_{j=t}^k \mathbb{P}\left[\abs{\inner{\normalized{P_{t-1}^{\perp}x^{(t)}}}{\normalized{P_{t-1}^\perp v_j}}} \geq \frac{c}{2(\sqrt{2}+\sqrt{k-1})} \ \middle|\ G_{<t}, V_{<t} \right]
\end{aligned}
\end{equation}
The first term in the inner product is fixed given $G_{<t}, V_{<t}$, and we showed above that the second term is a unit vector that is distributed spherically symmetrically on the unit sphere in $S_{t-1}^\perp$ given $G_{<t}, V_{<t}$. Therefore, each probability is equal to $\mathbb{P}\left(\inner{u}{e_1} \geq \frac{c}{2(\sqrt{2}+\sqrt{k-1})} \right)$ where $u$ is uniformly random on the unit sphere in $\mathbb{R}^{d'}$ where $d' = \dim(S_{t-1}^\perp) \geq k - 2t + 2$. 

Imagining a unit sphere with ``up" and ``down" corresponding to $\pm e_1$, $\mathbb{P}\left(\inner{u}{e_1} \geq \frac{c}{2(\sqrt{2}+\sqrt{k-1})} \right)$ is the surface area of the ``end caps" of the sphere lying above and below circles of radius $R := \sqrt{1 - \left(\frac{c}{2(\sqrt{2}+\sqrt{k-1})}\right)^2}$, which is strictly smaller than the surface area of a full sphere of radius $R$. Therefore,
\begin{equation}
\begin{aligned} 
&\mathbb{P}\left[\abs{\inner{\normalized{P_{t-1}^{\perp}x^{(t)}}}{\normalized{P_{t-1}^\perp v_j}}} \geq \frac{c}{2(\sqrt{2}+\sqrt{k-1})} \ \middle|\ G_{<t}, V_{<t} \right] \\
&< \frac{\textrm{SurfaceArea}_{d-2t+2}(R)}{\textrm{SurfaceArea}_{d-2t+2}(1)} \\
&= R^{d-2t + 1} \\
&= \left(1 - \left(\frac{c}{2(\sqrt{2}+\sqrt{k-1})}\right)^2 \right)^{\frac{d-2t+1}{2}} \\
&\leq \exp\left( -\left(\frac{c}{2(\sqrt{2}+\sqrt{k-1})}\right)^2\frac{d-2t+1}{2} \right) 
\end{aligned}
\end{equation}
With the final inequality coming from the fact that $1-x \leq \exp(-x)$.  This holds for each $j \geq t$, therefore,

\begin{equation}
\begin{aligned}
\mathbb{P}\left[ G_t\ \middle|\ G_{<t} \right]
&\geq \inf_{V_{<t}} \mathbb{P}\left[ G_t\ \middle|\ G_{<t}, V_{<t} \right] \\
&\geq 1 - (k-t+1) \exp\left( -\left(\frac{c}{2(\sqrt{2}+\sqrt{k-1})}\right)^2\frac{d-2t+1}{2} \right) \\
&= 1 - (k-t+1)\exp\left( \frac{-c^2(d-2t+1)}{8(\sqrt{2}+\sqrt{k-1})^2}\right) 
\end{aligned}
\end{equation}
\end{proof}

Finally, bringing together Lemma \ref{lem:1term} and \eqref{eq:chainrule}
\begin{lemma} \label{lem:alltogether}
For any $\epsilon \in (0,\frac{1}{2})$, and dimension $d \geq \frac{2}{\epsilon^8}\log\frac{1}{\epsilon^4}$, $\mathbb{P}\left[E\right] > \frac{15}{16}$, where the probability is over the random choice of $\set{v_j}$.
\end{lemma}
\begin{proof}
By Lemma \ref{lem:1term}, for all $t$
\[ \mathbb{P}\left[ G_t\ \middle|\ G_{<t} \right] > 1 - (k-t+1)\exp\left( \frac{-c^2(d-2t+1)}{8(\sqrt{2}+\sqrt{k-1})^2}\right) \]
Combining this with Equation \eqref{eq:chainrule}:
\begin{equation}
\begin{aligned}
\mathbb{P}\left[ E \right] &\geq \prod_{t=1}^k \mathbb{P}\left[ G_t \ \middle|\ G_{<t} \right] \\
&> \prod_{t=1}^k \left( 1 - (k-t+1)\exp\left( \frac{-c^2(d-2t+1)}{8(\sqrt{2}+\sqrt{k-1})^2}\right) \right) \\
&> \left( 1 - k\exp\left( \frac{-c^2(d - 2k + 1)}{40k}\right) \right)^k \\
&\geq 1 - k^2\exp\left( \frac{-c^2(d - 2k + 1)}{40k}\right)
\end{aligned}
\end{equation}
Thus, when $\epsilon < \frac{1}{2}$ and
$d \geq \frac{2}{\epsilon^8}\log\frac{1}{\epsilon^4} \geq \frac{1}{\epsilon^8}\log\frac{1}{\epsilon^4} + 2k -1$
then
$\mathbb{P}\left[ E \right] > \frac{15}{16}$
\end{proof}

Thus $E$ occurs with constant probability when the dimension is sufficiently large, and when $E$ does occur, the algorithm must make at least $k$ queries the subgradient oracle of $f$ in order to find an $\epsilon$-suboptimal solution. Thus the expected number of oracle queries for any deterministic algorithm on the specified distribution over $\mathcal{F}$ is at least $\frac{15k}{16} = \Omega(1/\epsilon^2)$, applying Yao's minimax principle completes the proof. 

\bibliographystyle{plainnat}
{\footnotesize \bibliography{singlefunction}}

\end{document}